\documentclass[a4paper, 11pt, english, reqno]{smfart}
\usepackage[english]{babel}
\usepackage[OT2,T1]{fontenc}
\usepackage{lmodern}
\usepackage[hidelinks]{hyperref}
\usepackage{amsmath,amsthm,amssymb,amsfonts}
\usepackage[a4paper,left=3.2cm,right=3.2cm,top=3.2cm,bottom=3.2cm]{geometry}
\numberwithin{equation}{section}
\title[Height of rational points on random Fano hypersurfaces]{Height of rational points on random Fano hypersurfaces}
\author{Pierre Le Boudec}
\subjclass{$11$D$45$, $11$G$50$, $14$G$05$}
\address{Departement Mathematik und Informatik \\ Fachbereich Mathematik \\ Spiegelgasse $1$ \\ $4051$ Basel \\ Switzerland}
\email{pierre.leboudec@unibas.ch}

\begin{document}

\newtheorem{lemma}{Lemma}
\newtheorem{theorem}{Theorem}
\newtheorem{corollary}{Corollary}
\newtheorem{proposition}{Proposition}
\newtheorem{conjecture}{Conjecture}

\newcommand{\vol}{\operatorname{vol}}
\newcommand{\D}{\mathrm{d}}
\newcommand{\rank}{\operatorname{rank}}
\newcommand{\Pic}{\operatorname{Pic}}
\newcommand{\Gal}{\operatorname{Gal}}
\newcommand{\meas}{\operatorname{meas}}
\newcommand{\Spec}{\operatorname{Spec}}
\newcommand{\eff}{\operatorname{eff}}
\newcommand{\rad}{\operatorname{rad}}
\newcommand{\sq}{\operatorname{sq}}
\newcommand{\tors}{\operatorname{tors}}
\newcommand{\Cl}{\operatorname{Cl}}
\newcommand{\Reg}{\operatorname{Reg}}
\newcommand{\Sel}{\operatorname{Sel}}
\newcommand{\corank}{\operatorname{corank}}
\newcommand{\an}{\operatorname{an}}
\newcommand{\glob}{\operatorname{global}}
\newcommand{\Span}{\operatorname{Span}}

\begin{abstract}
We investigate in a statistical fashion the smallest height of a rational point on a Fano hypersurface defined over the field of rational numbers. Along the way, we establish an average version of Manin's conjecture about the number of rational points of bounded height on Fano varieties for the complete family of Fano hypersurfaces of fixed degree and dimension.
\end{abstract}

\maketitle

\tableofcontents

\section{Introduction}

Let $d, n \geq 2$ be such that $n \geq d$. Let also $N_{d,n} = \binom{n+d}{d}$ denote the number of monomials of degree $d$ in $n+1$ variables. Ordering monomials lexicographically, degree $d$ hypersurfaces in $\mathbb{P}^n$ which are defined over $\mathbb{Q}$ are parametrized by
\begin{equation*}
\mathbb{V}_{d,n} = \mathbb{P}^{N_{d,n}-1}(\mathbb{Q}).
\end{equation*}
From now on, we thus view the above hypersurfaces as elements of $\mathbb{V}_{d,n}$. We note that the assumption $n \geq d$ implies that a generic element of $\mathbb{V}_{d,n}$ is a smooth Fano hypersurface.

We now introduce an exponential height $H : \mathbb{P}^n(\mathbb{Q}) \to \mathbb{R}_{> 0}$. In order to do so, for any
$N \geq 1$ we let $\mathbb{Z}_{\mathrm{prim}}^N$ be the set of $(c_1, \dots, c_N) \in \mathbb{Z}^N$ such that
$\gcd(c_1, \dots, c_N) = 1$, and we also let $|| \cdot ||$ denote the Euclidean norm in $\mathbb{R}^N$. For
$x \in \mathbb{P}^n(\mathbb{Q})$, we choose $\mathbf{x} = (x_0, \dots, x_n) \in \mathbb{Z}_{\mathrm{prim}}^{n+1}$ such that
$x = (x_0 : \dots : x_n)$ and we set
\begin{equation}
\label{Definition H}
H(x) = ||\mathbf{x}||^{n+1-d}.
\end{equation}
Note that if $V \in \mathbb{V}_{d,n}$ then $H$ is an anticanonical height on $V$.

The goal of this article is to investigate the smallest height of a rational point on a hypersurface $V$ as $V$ runs over the space
$\mathbb{V}_{d,n}$. Therefore, for $V \in \mathbb{V}_{d,n}$, we define
\begin{equation*}
\mathfrak{M}(V) =
\begin{cases}
\displaystyle{\min_{x \in V(\mathbb{Q})}} H(x), & \textrm{if } V(\mathbb{Q}) \neq \emptyset, \\
\infty, & \textrm{if } V(\mathbb{Q}) = \emptyset.
\end{cases}
\end{equation*}

We note here that it would be more natural to consider elements of $\mathbb{V}_{d,n}$ up to isomorphism, but this seems to be a very hard task. We thus content ourselves with the point of view which was adopted for instance by Poonen and Voloch \cite{MR2029869} and by Bhargava \cite{BhargavaCubics}.

We now introduce an ordering of the space $\mathbb{V}_{d,n}$ by defining the height of $V \in \mathbb{V}_{d,n}$ as $||\mathbf{a}_V||$, where $\mathbf{a}_V \in \mathbb{Z}_{\mathrm{prim}}^{N_{d,n}}$ is any of the two primitive coefficient vectors associated to $V$. Finally, for $A \geq 1$, we let
\begin{equation*}
\mathbb{V}_{d,n}(A) = \left\{ V \in \mathbb{V}_{d,n} : ||\mathbf{a}_V|| \leq A \right\}.
\end{equation*}

It is well known that establishing upper bounds for $\mathfrak{M}(V)$ in terms of $||\mathbf{a}_V||$ is a hard problem as this would give a positive answer to Hilbert's tenth problem for the field of rational numbers. In the case $d=2$, Cassels \cite{MR0069217} proved that for any $V \in \mathbb{V}_{2,n}$ such that $V(\mathbb{Q}) \neq \emptyset$, we have
\begin{equation}
\label{Cassels}
\mathfrak{M}(V) \ll ||\mathbf{a}_V||^{n(n-1)/2},
\end{equation}
and Kneser \cite{MR0081306} constucted an example proving that the exponent $n(n-1)/2$ is optimal. It is also worth noting that Browning and Dietmann \cite[Corollary $1$]{MR2396125} have showed that for generic quadratic forms this exponent can be slightly improved if
$n \geq 4$.

In the case $d > 2$, the results are unfortunately sparse. The interested reader is invited to refer to Masser's survey \cite{MR1975456}. We also note that more recently, in the case $d=3$ and $n \geq 16$, Browning, Dietmann and Elliott \cite[Theorem $2$]{MR2885594} have proved that for $V \in \mathbb{V}_{3,n}$, we have $\mathfrak{M}(V) \ll ||\mathbf{a}_V||^{360000}$.

This situation being clearly unsatisfactory, Elsenhans and Jahnel have undertaken a number of theoretical and numerical investigations for the families of diagonal quartic threefolds \cite{MR2333743}, general cubic surfaces \cite{MR2641182} and diagonal cubic surfaces
\cite{MR2651158, MR2676747}. Their results have led them to suggest that, if $V$ belongs to any of these three families of hypersurfaces and if $V(\mathbb{Q}) \neq \emptyset$, then we might have the upper bound
\begin{equation}
\label{Expectation}
\mathfrak{M}(V) \ll \frac1{\tau(V)^{1+\varepsilon}},
\end{equation}
for any fixed $\varepsilon > 0$ and where $\tau(V)$ denotes the Tamagawa number appearing in the definition of the constant introduced by Peyre (see \cite{MR1340296}) in the context of Manin's conjecture (see \cite{MR974910}). Moreover they have showed
\cite[Theorem $2$.$2$]{MR2333743} that the expectation \eqref{Expectation} with $\varepsilon = 0$ does not hold in general. Note finally that the quantities $1/\tau(V)$ and $||\mathbf{a}_V||$ are expected to have comparable size for most $V \in \mathbb{V}_{d,n}$ (see
\cite[Theorem $3$.$4$.$3$]{MR2333743} and \cite[Theorem, Section $1$.$5$]{MR2651158}).

The following theorem is our main result and shows that the orders of magnitude of $\mathfrak{M}(V)$ and $||\mathbf{a}_V||$ are indeed closely related.

\begin{theorem}
\label{Theorem 1}
Let $d \geq 2$ and $n \geq d$ with $(d,n) \neq (2,2)$. Let also $\psi : \mathbb{R}_{>0} \to \mathbb{R}_{>0}$ be such that
$\psi(u) = o(u)$ as $u \to \infty$. Then we have
\begin{equation*}
\lim_{A \to \infty} \frac{\# \left\{ V \in \mathbb{V}_{d,n}(A) : \mathfrak{M}(V) > \psi(||\mathbf{a}_V||) \right\}}{\# \mathbb{V}_{d,n}(A)} = 1.
\end{equation*}
\end{theorem}

Theorem~\ref{Theorem 1} is pertinent as we expect that a positive proportion of hypersurfaces in $\mathbb{V}_{d,n}$ admit a rational point for any $d \geq 2$ and $n \geq d$ with $(d,n) \neq (2,2)$. More precisely, a conjecture of Poonen and Voloch \cite[Conjecture $2$.$2$]{MR2029869} implies in particular that
\begin{equation*}
\liminf_{A \to \infty} \frac{\# \left\{ V \in \mathbb{V}_{d,n}(A) : V(\mathbb{Q}) \neq \emptyset \right\}}{\# \mathbb{V}_{d,n}(A)} > 0.
\end{equation*}
We note that this lower bound is proved in forthcoming work of Browning, the author and Sawin. Moreover, it transpires from this investigation that if $(d,n) \neq (3,3)$ then Theorem~\ref{Theorem 1} is optimal in the sense that for any $\eta >0$, we have
\begin{equation*}
\limsup_{A \to \infty}
\frac{\# \left\{ V \in \mathbb{V}_{d,n}(A) : \mathfrak{M}(V) > \eta ||\mathbf{a}_V|| \right\}}{\# \mathbb{V}_{d,n}(A)} < 1.
\end{equation*}

The only case which is not covered by Theorem~\ref{Theorem 1} is the case $(d,n)=(2,2)$ of conics. The situation is fundamentally different here as we have
\begin{equation*}
\frac1{(\log A)^{1/2}} \ll \frac{\# \left\{ V \in \mathbb{V}_{2,2}(A) : V(\mathbb{Q}) \neq \emptyset \right\}}{\# \mathbb{V}_{2,2}(A)} \ll \frac1{(\log A)^{1/2}}.
\end{equation*}
The upper bound and the lower bound have respectively been obtained by Serre \cite{MR1075658} and Hooley \cite {MR2300456}.

As already mentioned, Kneser showed that the exponent $n(n-1)/2$ in the upper bound \eqref{Cassels} could not be improved. The following theorem implies in particular that in the case $n=2$, the exponent $1$ is actually optimal for $100 \%$ of all conics admitting a rational point.

\begin{theorem}
\label{Theorem 1'}
Let $\psi : \mathbb{R}_{>0} \to \mathbb{R}_{>0}$ be such that $\psi(u)=o\left(u/(\log u)^{1/2}\right)$ as $u \to \infty$. Then we have
\begin{equation*}
\lim_{A \to \infty} \frac{\# \left\{ V \in \mathbb{V}_{2,2}(A) : \psi(||\mathbf{a}_V||) < \mathfrak{M}(V) \leq 3 ||\mathbf{a}_V|| \right\}}
{\# \left\{ V \in \mathbb{V}_{2,2}(A) : V(\mathbb{Q}) \neq \emptyset \right\}} = 1.
\end{equation*}
\end{theorem}

We remark that the analogous theorem for the family of diagonal conics can be established by appealing to a result of the author
\cite[Lemma $4$]{MR3456025} and the work of Hooley \cite{MR1199934}.

We finally state another result which is the key tool in the proofs of Theorems~\ref{Theorem 1} and~\ref{Theorem 1'} and which is definitely of independent interest. For $V \in \mathbb{V}_{d,n}$ and for any subset $U \subset V$, we define the counting function
\begin{equation}
\label{Definition N(B)}
N_U(B) = \# \left\{ x \in U(\mathbb{Q}) : H(x) \leq B \right\}.
\end{equation}
If $V \in \mathbb{V}_{d,n}$ is a smooth hypersurface such that $V(\mathbb{Q})$ is Zariski dense in $V$, a corrected version of Manin's conjecture \cite{MR974910} states that there should exist a thin subset $T \subset V$ such that
\begin{equation}
\label{Conjecture Manin}
N_{V \smallsetminus T}(B) = c_V B (\log B)^{\rho - 1} (1 + o(1)),
\end{equation}
where $c_V > 0$ is Peyre's constant (see \cite{MR1340296}) and $\rho$ denotes the rank of the Picard group of $V$.

The following theorem can naturally be viewed as an average version of Manin's conjecture over the space $\mathbb{V}_{d,n}$. We note that here and throughout Sections~\ref{Section Manin} and~\ref{Section smallest height} all the implied constants depend at most on $d$ and
$n$.

\begin{theorem}
\label{Theorem 2}
Let $d \geq 2$ and $n \geq d$. There exists a constant $C_{d,n} >0$ such that for $A \geq B^{1/(n+1-d)}$, we have
\begin{equation*}
\frac1{\# \mathbb{V}_{d,n}(A)} \sum_{ V \in \mathbb{V}_{d,n}(A)} N_V(B) =
C_{d,n} \frac{B}{A} \left( 1 + O \left( \frac{B^{1/(n+1-d)}}{A} + \frac{\log B}{B^{1/(n+1-d)}} \right) \right).
\end{equation*}
\end{theorem}

Since the rank of the Picard group of a generic element of $\mathbb{V}_{d,n}$ is equal to $1$, we see that Theorem~\ref{Theorem 2} agrees with the expectation \eqref{Conjecture Manin}. In addition, we note that the constant $C_{d,n}$ is explicit and its expression is given at the end of Section~\ref{Section Manin}. Also, we remark that the factor $\log B$ in the error term is only necessary if $n = d$.

The analog of Theorem~\ref{Theorem 2} in the case $d \geq 3$ and $n = d-1$ of intermediate type hypersurfaces is worth of attention. If we replace the exponent $n+1-d$ by $1$ in the definition \eqref{Definition H} of the height function $H$ then our methods show that there exists a constant $c_d >0$ such that for $A \geq B$, we have
\begin{equation*}
\frac1{\# \mathbb{V}_{d,d-1}(A)} \sum_{ V \in \mathbb{V}_{d,d-1}(A)} N_V(B)
= c_d \frac{\log B}{A} \left( 1 + O \left( \frac{B}{A} + \frac1{\log B} \right) \right).
\end{equation*}
It is interesting to note that in the case $d=4$, this estimate is in agreement with the unpublished observation of van Luijk that the number of rational points of height bounded by $B$ on a K$3$ surface should, in favourable circumstances, asymptotically behave like
$(\log B)^{\rho}$.

Let us give a quick sketch of the proof of Theorem~\ref{Theorem 2}. The first step consists in translating the statement into lattice point counting problems. Then, for each lattice to be considered, an astute combinatorial argument enables us to explicitly construct a maximal family of linearly independent vectors. Moreover, we crucially observe that we can control the norms of these vectors. Finally, we check that feeding this observation into classical geometry of numbers results allows us to complete the proof.

The topics addressed in this article have not been much studied. However, Br\"{u}dern and Dietmann have proved
\cite[Theorem $1$.$4$]{MR3177289} the analog of Theorem~\ref{Theorem 1} for families of diagonal hypersurfaces under the assumption $n \geq 2d-1$, while Theorem~\ref{Theorem 1} holds in the entire Fano range $n \geq d$.

We conclude this introduction by describing what we can achieve for families of Fano complete intersections. We let $r > 1$ and
$d_1, \dots, d_r \geq 2$ and we assume that $n \geq d_1 + \cdots + d_r$. Intersections of $r$ hypersurfaces in $\mathbb{P}^n$ of respective degrees $d_1, \dots, d_r$ and which are defined over $\mathbb{Q}$ are parametrized by
\begin{equation*}
\mathbb{V}_{\mathbf{d},n} = \mathbb{P}^{N_{d_1,n}-1}(\mathbb{Q}) \times \cdots \times \mathbb{P}^{N_{d_r,n}-1}(\mathbb{Q}),
\end{equation*}
where $\mathbf{d} = (d_1, \dots, d_r)$. Using the Segre embedding it is thus natural to define the height of
$V = (V_1, \dots, V_r) \in \mathbb{V}_{\mathbf{d},n}$ as $||\mathbf{a}_{V_1}|| \cdots ||\mathbf{a}_{V_r}||$. Unfortunately, this definition seems to create insurmountable difficulties as one faces situations in which $V$ lies in cuspidal regions of the space
$\mathbb{V}_{\mathbf{d},n}$. However, our techniques show that the analogs of Theorems~\ref{Theorem 1} and~\ref{Theorem 2} hold if
$n \geq r-1+d_1 + \cdots + d_r$ provided that we define the height of $V$ as $\max \{ ||\mathbf{a}_{V_i}|| : i \in \{1, \dots, r \} \}$.

\subsection*{Acknowledgements}

It is a pleasure for the author to thank Philipp Habegger, Philippe Michel and Ramon Moreira Nunes for interesting conversations. In addition, the author is extremely grateful to Tim Browning for his very careful reading of an earlier version of the manuscript, and for making a number of comments which greatly improved the presentation and the content of this article. This work was initiated while the author was working as an Instructor at the \'{E}cole Polytechnique F\'{e}d\'{e}rale de Lausanne. The financial support and the wonderful working conditions that the author enjoyed during the four years he worked at this institution are gratefully acknowledged. The research of the author is integrally funded by the Swiss National Science Foundation through the SNSF Professorship number $170565$ awarded to the project
\textit{Height of rational points on algebraic varieties}. Both the financial support of the SNSF and the perfect working conditions provided by the University of Basel are gratefully acknowledged.

\section{Geometry of numbers}

\subsection{Lattice point counting estimates}

We start by setting some notation and recalling several classical definitions. For $N \geq 1$, we let $\langle \cdot, \cdot \rangle$ be the Euclidean inner product in $\mathbb{R}^N$ and we recall that $|| \cdot ||$ denotes the Euclidean norm in $\mathbb{R}^N$. Also, for
$u > 0$, we set
\begin{equation*}
\mathcal{B}_N(u) = \left\{ \mathbf{y} \in \mathbb{R}^{N} : ||\mathbf{y}|| \leq u \right\},
\end{equation*}
and we let $V_N$ denote the volume of the unit ball $\mathcal{B}_N(1)$ in $\mathbb{R}^N$.

A lattice $\Lambda \subset \mathbb{R}^N$ is a discrete subgroup of $\mathbb{R}^N$. If $\Lambda$ is a lattice then the dimension of the subspace $\Span_{\mathbb{R}} (\Lambda)$ is called the rank of $\Lambda$. Also, if $\Lambda$ is a lattice of rank $R$ then its determinant
$\det (\Lambda)$ is defined as the $R$-dimensional volume of the fundamental parallelepiped spanned by any basis of $\Lambda$. Therefore, if $(\mathbf{b}_1, \dots, \mathbf{b}_R)$ is a basis of $\Lambda$ and $\mathbf{B}$ is the $N \times R$ matrix whose columns are the vectors $\mathbf{b}_1, \dots, \mathbf{b}_R$ then
\begin{equation}
\label{Definition det}
\det (\Lambda) = \sqrt{ \det (\mathbf{B}^T \mathbf{B})}.
\end{equation} 

In addition, if a lattice $\Lambda \subset \mathbb{R}^N$ is a subset of $\mathbb{Z}^N$ then $\Lambda$ is said to be integral. Moreover, we say that an integral lattice $\Lambda$ is primitive if it is equal to the maximal integral lattice in $\Span_{\mathbb{Q}} (\Lambda)$, that is if
$\Lambda = \Span_{\mathbb{Q}} (\Lambda) \cap \mathbb{Z}^N$.

Given a rank $R$ lattice $\Lambda \subset \mathbb{R}^N$, we let $\lambda_1(\Lambda), \dots, \lambda_R(\Lambda)$ denote its successive minima with respect to the unit ball $\mathcal{B}_N(1)$. We recall that they are defined for $i \in \{1, \dots, R\}$ by
\begin{equation}
\label{Definition successive}
\lambda_i(\Lambda) = \inf \left\{ u \in \mathbb{R}_{> 0} : \dim ( \Span_{\mathbb{R}} ( \Lambda \cap \mathcal{B}_N(u) ) ) \geq i \right\}.
\end{equation}
We clearly have $\lambda_1(\Lambda) \leq \dots \leq \lambda_R(\Lambda)$ and the celebrated second Theorem of Minkowski (see for instance \cite[Chapter VIII, Theorem V]{MR1434478}) states in particular that
\begin{equation}
\label{Estimate Minkowski}
\det (\Lambda) \leq \lambda_1(\Lambda) \cdots \lambda_R(\Lambda) \ll \det (\Lambda),
\end{equation}
where the implied constant depends at most on $R$.

Finally, for any lattice $\Lambda \subset \mathbb{R}^N$ and any bounded region $\mathcal{R} \subset \mathbb{R}^N$, we set
\begin{equation*}
\mathcal{N}(\Lambda; \mathcal{R}) = \# (\Lambda \cap \mathcal{R}),
\end{equation*}
and if $\Lambda$ is integral, we also set
\begin{equation}
\label{Definition N star}
\mathcal{N}^{\ast}(\Lambda; \mathcal{R}) = \# \left( \Lambda \cap \mathbb{Z}_{\mathrm{prim}}^N \cap \mathcal{R} \right).
\end{equation}

We use the convention that empty products and empty summations are respectively equal to $1$ and $0$. The three following lemmas build on a classical result of Schmidt \cite[Lemma~$2$]{MR0224562} which provides a handy estimate for the number of lattice points lying in a Euclidean ball.

\begin{lemma}
\label{Lemma 1}
Let $N \geq 1$ and $R \in \{1, \dots, N\}$. Let $\Lambda \subset \mathbb{R}^N$ be a lattice of rank $R$. Let also $Y>0$ be such that the ball $\mathcal{B}_N(Y)$ contains $R$ linearly independent vectors of the lattice $\Lambda$. For $T \geq Y$, we have
\begin{equation*}
\mathcal{N}(\Lambda; \mathcal{B}_N(T)) = V_R \frac{T^R}{\det (\Lambda)} \left( 1 + O \left( \frac{Y}{T} \right) \right),
\end{equation*}
where the implied constant depends at most on $R$.
\end{lemma}

\begin{proof}
The result of Schmidt \cite[Lemma $2$]{MR0224562} states that
\begin{equation}
\label{Schmidt}
\mathcal{N}(\Lambda; \mathcal{B}_N(T)) = V_R \frac{T^R}{\det (\Lambda)} +
O \left( \sum_{i=1}^R \frac{T^{R-i}}{\lambda_1(\Lambda) \cdots \lambda_{R-i}(\Lambda)} \right).
\end{equation}
Using Minkowski's estimate \eqref{Estimate Minkowski} we thus deduce
\begin{equation*}
\mathcal{N}(\Lambda; \mathcal{B}_N(T)) = V_R \frac{T^R}{\det (\Lambda)}
\left( 1 + O \left( \sum_{i=1}^R \frac{\lambda_{R-i+1}(\Lambda) \cdots \lambda_{R}(\Lambda)}{T^i} \right) \right).
\end{equation*}
Recalling the definition \eqref{Definition successive} of the successive minima, we see that the assumption that the ball $\mathcal{B}_N(Y)$ contains $R$ linearly independent vectors of the lattice $\Lambda$ means exactly that we have $\lambda_i(\Lambda) \leq Y$ for any
$i \in \{1, \dots, R \}$. Therefore, we obtain
\begin{equation*}
\mathcal{N}(\Lambda; \mathcal{B}_N(T)) = V_R \frac{T^R}{\det (\Lambda)} \left( 1 + O \left( \sum_{i=1}^R \frac{Y^i}{T^i} \right) \right),
\end{equation*}
and the assumption $T \geq Y$ allows us to complete the proof.
\end{proof}

The following result is concerned with integral lattices. We note that in the present work we only use it in the case $R_0=R-2$ but this general version may be useful in other contexts.

\begin{lemma}
\label{Lemma 2}
Let $N \geq 1$ and $R \in \{1, \dots, N\}$. Let $\Lambda \subset \mathbb{R}^N$ be an integral lattice of rank $R$. Let also $Y \geq 1$ be such that the ball $\mathcal{B}_N(Y)$ contains $R$ linearly independent vectors of the lattice $\Lambda$. For any
$R_0 \in \{0, \dots, R-1 \}$ and $T \leq Y$, we have
\begin{equation*}
\mathcal{N}(\Lambda; \mathcal{B}_N(T)) \ll T^{R-R_0-1} \left( \frac{T Y^{R_0}}{\det (\Lambda)} + 1 \right),
\end{equation*}
where the implied constant depends at most on $R$.
\end{lemma}

\begin{proof}
We start by noting that the estimate \eqref{Schmidt} implies in particular that
\begin{equation*}
\mathcal{N}(\Lambda; \mathcal{B}_N(T)) \ll
\frac{T^R}{\det (\Lambda)} + \sum_{i=1}^R \frac{T^{R-i}}{\lambda_1(\Lambda) \cdots \lambda_{R-i}(\Lambda)}.
\end{equation*}
Using Minkowski's estimate \eqref{Estimate Minkowski} we see that for any $R_0 \in \{0, \dots, R-1 \}$, we have
\begin{equation*}
\mathcal{N}(\Lambda; \mathcal{B}_N(T)) \ll
\frac{T^R}{\det (\Lambda)} \left( 1 + \sum_{i=1}^{R_0} \frac{\lambda_{R-i+1}(\Lambda) \cdots \lambda_R(\Lambda)}{T^i} \right) + \sum_{i=R_0+1}^R \frac{T^{R-i}}{\lambda_1(\Lambda) \cdots \lambda_{R-i}(\Lambda)}.
\end{equation*}
By assumption, the lattice $\Lambda$ is integral so we have $\lambda_i(\Lambda) \geq 1$ for any $i \in \{1, \dots, R \}$. In addition, as already noted in the proof of Lemma~\ref{Lemma 1}, the assumption that the ball $\mathcal{B}_N(Y)$ contains $R$ linearly independent vectors of the lattice $\Lambda$ amounts to saying that $\lambda_i(\Lambda) \leq Y$ for any $i \in \{1, \dots, R \}$. As a result, we deduce
\begin{equation*}
\mathcal{N}(\Lambda; \mathcal{B}_N(T)) \ll \frac{T^R}{\det (\Lambda)} \left( 1 + \sum_{i=1}^{R_0} \frac{Y^i}{T^i} \right) +
\sum_{i=R_0+1}^R T^{R-i}.
\end{equation*}
Since $T \leq Y$, this finally gives
\begin{equation*}
\mathcal{N}(\Lambda; \mathcal{B}_N(T)) \ll \frac{T^{R-R_0} Y^{R_0}}{\det (\Lambda)} + T^{R-R_0-1},
\end{equation*}
which completes the proof.
\end{proof}

The next result applies to primitive lattices and will be particularly useful in the proof of Theorem~\ref{Theorem 2}.

\begin{lemma}
\label{Lemma 3}
Let $N \geq 3$ and $R \in \{3, \dots, N\}$. Let $\Lambda \subset \mathbb{R}^N$ be an integral and primitive lattice of rank $R$. Let also $Y \geq 1$ be such that the ball $\mathcal{B}_N(Y)$ contains $R$ linearly independent vectors of the lattice $\Lambda$. For $T \geq Y$, we have
\begin{equation*}
\mathcal{N}^{\ast}(\Lambda; \mathcal{B}_N(T)) =
\frac{V_R}{\zeta(R)} \cdot \frac{T^R}{\det (\Lambda)} \left( 1 + O \left( \frac{Y}{T} \right) \right) + O (T \log Y),
\end{equation*}
where the implied constants depend at most on $R$.
\end{lemma}

\begin{proof}
A M\"{o}bius inversion gives
\begin{equation*}
\mathcal{N}^{\ast}(\Lambda; \mathcal{B}_N(T)) = \sum_{\ell \leq T}
\mu(\ell) \left( \mathcal{N} \left( \left( \frac1{\ell} \cdot \Lambda \right) \cap \mathbb{Z}^N; \mathcal{B}_N \left( \frac{T}{\ell} \right) \right) - 1 \right),
\end{equation*}
where the term $-1$ accounts for the zero vector. Since $\Lambda$ is a primitive lattice, for any integer $ \ell \geq 1$ we have
\begin{equation*}
\left( \frac1{\ell} \cdot \Lambda \right) \cap \mathbb{Z}^N = \Lambda.
\end{equation*}
Splitting the summation over $\ell$ into two different summations depending on whether $\ell \leq T/Y$ or $T/Y < \ell \leq T$, we thus obtain
\begin{equation}
\label{Separate}
\begin{split}
\mathcal{N}^{\ast}(\Lambda; \mathcal{B}_N(T)) = & \ 
\sum_{\ell \leq T/Y} \mu(\ell) \mathcal{N} \left(\Lambda; \mathcal{B}_N \left(\frac{T}{\ell} \right) \right) \\
& + \sum_{T/Y < \ell \leq T} \mu(\ell) \mathcal{N} \left(\Lambda; \mathcal{B}_N \left(\frac{T}{\ell} \right) \right) + O(T).
\end{split}
\end{equation}

First, it follows from Lemma~\ref{Lemma 1} that
\begin{equation*}
\sum_{\ell \leq T/Y} \mu(\ell) \mathcal{N} \left(\Lambda; \mathcal{B}_N \left(\frac{T}{\ell} \right) \right) =
\sum_{\ell \leq T/Y} \mu(\ell) V_R \frac{T^R}{\ell^R \det (\Lambda)} \left( 1 + O \left( \frac{\ell Y}{T} \right) \right).
\end{equation*}
Since $R \geq 3$, we deduce
\begin{equation}
\label{Equation 1}
\sum_{\ell \leq T/Y} \mu(\ell) \mathcal{N} \left(\Lambda; \mathcal{B}_N \left(\frac{T}{\ell} \right) \right) =
\frac{V_R}{\zeta(R)} \cdot \frac{T^R}{\det (\Lambda)} \left( 1 + O \left( \frac{Y}{T} \right) \right).
\end{equation}
Next, an application of Lemma~\ref{Lemma 2} with $R_0=R-2$ yields
\begin{equation*}
\sum_{T/Y < \ell \leq T} \mu(\ell) \mathcal{N} \left(\Lambda; \mathcal{B}_N \left(\frac{T}{\ell} \right) \right) \ll
\sum_{T/Y < \ell \leq T} \left( \frac{T^2 Y^{R-2}}{\ell^2 \det (\Lambda)} + \frac{T}{\ell} \right).
\end{equation*}
Therefore, we get
\begin{equation*}
\sum_{T/Y < \ell \leq T} \mu(\ell) \mathcal{N} \left( \Lambda; \mathcal{B}_N \left(\frac{T}{\ell} \right) \right) \ll
\frac{T Y^{R-1}}{\det (\Lambda)} + T \log Y,
\end{equation*}
and since $T \geq Y$, this gives
\begin{equation}
\label{Equation 2}
\sum_{T/Y < \ell \leq T} \mu(\ell) \mathcal{N} \left( \Lambda; \mathcal{B}_N \left(\frac{T}{\ell} \right) \right) \ll
\frac{T^{R-1} Y}{\det (\Lambda)} + T \log Y.
\end{equation}
Combining the estimates \eqref{Separate} and \eqref{Equation 1} and the upper bound \eqref{Equation 2} allows us to complete the proof.
\end{proof}

\subsection{The determinant of orthogonal lattices}

If $\Lambda \subset \mathbb{R}^N$ is an integral lattice we define the lattice $\Lambda^{\perp}$ orthogonal to $\Lambda$ by
\begin{equation*}
\Lambda^{\perp} = \left\{ \mathbf{y} \in \mathbb{Z}^N :
\forall \mathbf{z} \in \Lambda \ \langle \mathbf{z}, \mathbf{y} \rangle = 0 \right\}.
\end{equation*}
In other words, $\Lambda^{\perp}$ is the set of all vectors with integral coordinates in the orthogonal complement of
$\Span_{\mathbb{Q}}(\Lambda)$. The lattice $\Lambda^{\perp}$ is thus primitive by definition. Moreover, if $\Lambda$ is an integral lattice of rank $R$ then $\Lambda^{\perp}$ has rank $N-R$. Indeed, if $(\mathbf{b}_1, \dots, \mathbf{b}_{N-R})$ is a $\mathbb{Q}$-basis of the orthogonal complement of $\Span_{\mathbb{Q}}(\Lambda)$ such that $\mathbf{b}_1, \dots, \mathbf{b}_{N-R} \in \mathbb{Z}^N$ then $\Lambda^{\perp}$ contains $\mathbb{Z} \mathbf{b}_1 \oplus \dots \oplus \mathbb{Z} \mathbf{b}_{N-R}$. Finally, we remark that if $\Lambda$ is a primitive lattice (see for instance \cite[Corollary of Lemma~$1$]{MR0224562}) then
\begin{equation}
\label{Equality det orthogonal}
\det (\Lambda^{\perp}) = \det (\Lambda).
\end{equation}

For $\mathbf{c} \in \mathbb{Z}^N$, it is convenient to set $\Lambda_{\mathbf{c}} = (\mathbb{Z} \mathbf{c})^{\perp}$, that is
\begin{equation*}
\Lambda_{\mathbf{c}} = \left\{ \mathbf{y} \in \mathbb{Z}^N : \langle \mathbf{c}, \mathbf{y} \rangle = 0 \right\}.
\end{equation*}
If $\Lambda$ is an integral lattice with a specified basis then the following lemma provides us with a formula for the determinant of the lattice
$\Lambda^{\perp}$ in terms of the given basis of $\Lambda$. We note that in this work we only use the case where $\Lambda$ has rank $1$ but this general result will be useful in future works.

\begin{lemma}
\label{Lemma 4}
Let $N \geq 1$ and $k \in \{1, \dots, N-1\}$. Let also $\mathbf{c}_1, \dots, \mathbf{c}_k \in \mathbb{Z}^N$ be linearly independent vectors. Then $\Lambda_{\mathbf{c}_1} \cap \dots \cap \Lambda_{\mathbf{c}_k}$ is a primitive lattice of rank $N - k$. Moreover, we have
\begin{equation*}
\det ( \Lambda_{\mathbf{c}_1} \cap \dots \cap \Lambda_{\mathbf{c}_k} ) =
\frac{\det (\mathbb{Z} \mathbf{c}_1 \oplus \dots \oplus \mathbb{Z} \mathbf{c}_k)}
{\mathcal{G}(\mathbf{c}_1, \dots, \mathbf{c}_k)},
\end{equation*}
where $\mathcal{G}(\mathbf{c}_1, \dots, \mathbf{c}_k)$ denotes the greatest common divisor of the $k \times k$ minors of the
$N \times k$ matrix whose columns are the vectors $\mathbf{c}_1, \dots, \mathbf{c}_k$.
\end{lemma}

\begin{proof}
By assumption, the lattice $\mathbb{Z} \mathbf{c}_1 \oplus \dots \oplus \mathbb{Z} \mathbf{c}_k$ has rank $k$, so the obvious equality
\begin{equation*}
( \mathbb{Z} \mathbf{c}_1 \oplus \dots \oplus \mathbb{Z} \mathbf{c}_k)^{\perp} =
\Lambda_{\mathbf{c}_1} \cap \dots \cap \Lambda_{\mathbf{c}_k}
\end{equation*}
shows that $\Lambda_{\mathbf{c}_1} \cap \dots \cap \Lambda_{\mathbf{c}_k}$ is a primitive lattice of rank $N-k$.

Therefore, the equality \eqref{Equality det orthogonal} gives
\begin{equation}
\label{Equality determinants}
\det ((\Lambda_{\mathbf{c}_1} \cap \dots \cap \Lambda_{\mathbf{c}_k})^{\perp}) =
\det ( \Lambda_{\mathbf{c}_1} \cap \dots \cap \Lambda_{\mathbf{c}_k}).
\end{equation}
Let $(\mathbf{d}_1, \dots, \mathbf{d}_k)$ be a basis of the lattice
$(\Lambda_{\mathbf{c}_1} \cap \dots \cap \Lambda_{\mathbf{c}_k})^{\perp}$. Let also $\mathbf{C}$ and
$\mathbf{D}$ be the $N \times k$ matrices whose columns are respectively the vectors $\mathbf{c}_1, \dots, \mathbf{c}_k$ and
$\mathbf{d}_1, \dots, \mathbf{d}_k$. Since the lattice $\mathbb{Z} \mathbf{c}_1 \oplus \dots \oplus \mathbb{Z} \mathbf{c}_k$ is a sublattice of $(\Lambda_{\mathbf{c}_1} \cap \dots \cap \Lambda_{\mathbf{c}_k})^{\perp}$, there exists a $k \times k$ matrix
$\mathbf{M}$ with integral entries such that $\mathbf{C} = \mathbf{D} \mathbf{M}$. Recalling the definition \eqref{Definition det} of the determinant of a lattice, we see that
\begin{align*}
\det (\mathbb{Z} \mathbf{c}_1 \oplus \dots \oplus \mathbb{Z} \mathbf{c}_k) & = \sqrt{\det (\mathbf{C}^T\mathbf{C})} \\
& = |\det(\mathbf{M})| \cdot \sqrt{\det (\mathbf{D}^T\mathbf{D})} \\
& = |\det(\mathbf{M})| \cdot \det ((\Lambda_{\mathbf{c}_1} \cap \dots \cap \Lambda_{\mathbf{c}_k})^{\perp}).
\end{align*}
As a result, the equality \eqref{Equality determinants} yields
\begin{equation}
\label{Equality M}
\det ( \Lambda_{\mathbf{c}_1} \cap \dots \cap \Lambda_{\mathbf{c}_k} ) =
\frac{\det (\mathbb{Z} \mathbf{c}_1 \oplus \dots \oplus \mathbb{Z} \mathbf{c}_k)}{|\det(\mathbf{M})|}.
\end{equation}

Now, for any $N \times k$ matrix $\mathbf{F}$ and any $(i_1, \dots, i_{N-k}) \in \{1, \dots, N \}^{N-k}$ satisfying
$i_1 < \dots < i_{N-k}$, we let $[\mathbf{F}]_{(i_1, \dots, i_{N-k})}$ be the $k \times k$ matrix formed by removing from $\mathbf{F}$ its rows of indices $i_1, \dots, i_{N-k}$. Since we have
\begin{equation*}
[\mathbf{C}]_{(i_1, \dots, i_{N-k})} = [\mathbf{D}]_{(i_1, \dots, i_{N-k})} \cdot \mathbf{M},
\end{equation*}
we deduce that
\begin{equation}
\label{Equality minors}
\det ( [\mathbf{C}]_{(i_1, \dots, i_{N-k})} ) = \det ([\mathbf{D}]_{(i_1, \dots, i_{N-k})}) \cdot \det (\mathbf{M}).
\end{equation}
But the lattice $(\Lambda_{\mathbf{c}_1} \cap \dots \cap \Lambda_{\mathbf{c}_k})^{\perp}$ is primitive so
\cite[Chapter I, Corollary~$3$]{MR1434478} implies that its basis $(\mathbf{d}_1, \dots, \mathbf{d}_k)$ can be extended into a basis of
$\mathbb{Z}^N$. Therefore, it follows from \cite[Chapter I, Lemma $2$]{MR1434478} that
$\mathcal{G}(\mathbf{d}_1, \dots, \mathbf{d}_k) = 1$ and the equality \eqref{Equality minors} gives
\begin{equation*}
\mathcal{G}(\mathbf{c}_1, \dots, \mathbf{c}_k) = |\det (\mathbf{M})|.
\end{equation*}
Recalling the equality \eqref{Equality M}, we see that this completes the proof.
\end{proof}

\subsection{The key combinatorial argument}

For $d \geq 1$ and $n \geq 0$, we introduce the Veronese embedding $\nu_{d,n} : \mathbb{R}^{n+1} \to \mathbb{R}^{N_{d,n}}$ defined by listing all the monomials of degree $d$ in $n+1$ variables using the lexicographical ordering. The following lemma is the key tool in the proof of Theorem~\ref{Theorem 2}.

\begin{lemma}
\label{Lemma 5}
Let $d \geq 1$ and $n \geq 0$. Let also $\mathbf{x} \in \mathbb{Z}_{\mathrm{prim}}^{n+1}$. Then the ball
$\mathcal{B}_{N_{d,n}}(||\mathbf{x}||)$ contains $N_{d,n}-1$ linearly independent vectors of the lattice
$\Lambda_{\nu_{d,n}(\mathbf{x})}$.
\end{lemma}

\begin{proof}
We proceed by induction on $n$. If $n=0$, the result is clear. Assume now that for some $n \geq 1$ the result holds for the integer $n-1$, and let $(x_0, \dots, x_n)$ denote the coordinates of the vector $\mathbf{x}$.

We first treat the case where $x_0 = 0$. We observe that the first $N_{d,n}-N_{d,n-1}$ coordinates of the vector $\nu_{d,n}(\mathbf{x})$ are equal to $0$. Therefore, the first $N_{d,n}-N_{d,n-1}$ vectors of the canonical basis of $\mathbb{R}^{N_{d,n}}$ belong to the lattice
$\Lambda_{\nu_{d,n}(\mathbf{x})}$. Moreover, they also belong to the ball $\mathcal{B}_{N_{d,n}}(||\mathbf{x}||)$ since
$\mathbf{x} \neq \boldsymbol{0}$. Let $E$ be the subspace of $\mathbb{R}^{N_{d,n}}$ defined by the vanishing of the first
$N_{d,n}-N_{d,n-1}$ coordinates. The induction hypothesis applied to the lattice $\Lambda_{\nu_{d,n}(\mathbf{x})} \cap E$ provides us with $N_{d,n-1} -1$ linearly independent vectors of this lattice in the ball $\mathcal{B}_{N_{d,n}}(||\mathbf{x}||)$. We thus obtain a family of $N_{d,n}-1$ linearly independent vectors of the lattice $\Lambda_{\nu_{d,n}(\mathbf{x})}$ belonging to the ball
$\mathcal{B}_{N_{d,n}}(||\mathbf{x}||)$, as desired.

We now deal with the case where $x_0 \neq 0$. For any monomial $P(T_0, \dots, T_n)$ different from $T_0^d$, we let
$i \in \{1, \dots, n\}$ be the least integer such that $T_i \mid P(T_0, \dots, T_n)$ and we define
\begin{equation*}
Q(T_0, \dots, T_n) = \frac{P(T_0, \dots, T_n)}{T_i} T_0.
\end{equation*}
It is crucial to note that $Q$ occurs before $P$ in the lexicographical ordering. We then let $\mathbf{v}(P) \in \mathbb{Z}^{N_{d,n}}$ be the vector with all coordinates equal to $0$, except the $Q$-coordinate equal to $-x_i$ and the $P$-coordinate equal to $x_0$. By construction, we have
\begin{equation*}
- x_i Q(\mathbf{x}) + x_0 P(\mathbf{x}) = 0,
\end{equation*}
which implies that $\mathbf{v}(P) \in \Lambda_{\nu_{d,n}(\mathbf{x})}$. Moreover, we clearly have $||\mathbf{v}(P)|| \leq ||\mathbf{x}||$ and thus $\mathbf{v}(P) \in \mathcal{B}_{N_{d,n}}(||\mathbf{x}||)$. Since we have associated a vector $\mathbf{v}(P)$ to each monomial $P$ except $T_0^d$, we see that we have constructed a family of $N_{d,n}-1$ vectors of the lattice $\Lambda_{\nu_{d,n}(\mathbf{x})}$ belonging to the ball $\mathcal{B}_{N_{d,n}}(||\mathbf{x}||)$. In addition, since $x_0 \neq 0$, the matrix of this family of vectors is in echelon form and so they are linearly independent. This completes the proof.
\end{proof}

\section{An average version of Manin's conjecture}

\label{Section Manin}

We now have all the tools we need to establish Theorem~\ref{Theorem 2}.

\begin{proof}[Proof of Theorem~\ref{Theorem 2}]
Recall the definition \eqref{Definition N(B)} of the counting function $N_V(B)$. We let
\begin{equation*}
S_{d,n}(A,B) = \sum_{ V \in \mathbb{V}_{d,n}(A)} N_V(B),
\end{equation*}
and we note that
\begin{equation*}
S_{d,n}(A,B) = \sum_{\substack{x \in \mathbb{P}^n(\mathbb{Q}) \\ H(x) \leq B}}
\# \left\{ V \in \mathbb{V}_{d,n}(A) : x \in V(\mathbb{Q}) \right\}.
\end{equation*}
It is convenient to set
\begin{equation*}
\Xi_{d,n}(B) = \left\{ \mathbf{x} \in \mathbb{Z}_{\mathrm{prim}}^{n+1} : ||\mathbf{x}|| \leq B^{1/(n+1-d)} \right\}.
\end{equation*}
Recalling the definition \eqref{Definition N star} of the quantity
$\mathcal{N}^{\ast}(\Lambda_{\nu_{d,n}(\mathbf{x})}; \mathcal{B}_{N_{d,n}}(A))$, we remark that
\begin{equation*}
S_{d,n}(A,B) =
\frac1{4} \sum_{\mathbf{x} \in \Xi_{d,n}(B)} \mathcal{N}^{\ast}(\Lambda_{\nu_{d,n}(\mathbf{x})}; \mathcal{B}_{N_{d,n}}(A)).
\end{equation*}

We have $||\mathbf{x}|| \leq B^{1/(n+1-d)}$, so Lemma~\ref{Lemma 5} shows that the ball
$\mathcal{B}_{N_{d,n}} \left(B^{1/(n+1-d)}\right)$ contains $N_{d,n}-1$ linearly independent vectors of the lattice
$\Lambda_{\nu_{d,n}(\mathbf{x})}$. In addition, $\Lambda_{\nu_{d,n}(\mathbf{x})}$ is a primitive lattice and by assumption
$A \geq B^{1/(n+1-d)}$. We can thus apply Lemma~\ref{Lemma 3} to deduce that
\begin{align*}
\mathcal{N}^{\ast}(\Lambda_{\nu_{d,n}(\mathbf{x})}; \mathcal{B}_{N_{d,n}}(A)) = & \ \frac{V_{N_{d,n}-1}}{\zeta(N_{d,n}-1)} \cdot
\frac{A^{N_{d,n}-1}}{\det (\Lambda_{\nu_{d,n}(\mathbf{x})})} \left( 1 + O \left( \frac{B^{1/(n+1-d)}}{A} \right) \right) \\
& + O (A \log B).
\end{align*}
Moreover, using the fact that $\mathbf{x} \in \mathbb{Z}_{\mathrm{prim}}^{n+1}$ we see that Lemma~\ref{Lemma 4} gives
\begin{equation*}
\det (\Lambda_{\nu_{d,n}(\mathbf{x})}) = ||\nu_{d,n}(\mathbf{x})||,
\end{equation*}
and it follows that
\begin{equation*}
\det (\Lambda_{\nu_{d,n}(\mathbf{x})}) \ll B^{d/(n+1-d)}.
\end{equation*}
Hence, our assumption $A \geq B^{1/(n+1-d)}$ implies in particular that
\begin{equation*}
\frac{A^{N_{d,n}-1}}{\det (\Lambda_{\nu_{d,n}(\mathbf{x})})} \cdot \frac{B^{1/(n+1-d)}}{A} \gg A \log B.
\end{equation*}
As a result, we obtain
\begin{equation*}
S_{d,n}(A,B) = \frac{V_{N_{d,n}-1}}{4\zeta(N_{d,n}-1)} A^{N_{d,n}-1} T_{d,n}(B)
\left( 1 + O \left( \frac{B^{1/(n+1-d)}}{A} \right) \right),
\end{equation*}
where we have set
\begin{equation*}
T_{d,n}(B) = \sum_{\mathbf{x} \in \Xi_{d,n}(B)} \frac1{||\nu_{d,n}(\mathbf{x})||}.
\end{equation*}
Furthermore, an elementary application of Lemma~\ref{Lemma 3} yields
\begin{equation*}
\# \mathbb{V}_{d,n}(A) = \frac{V_{N_{d,n}}}{2 \zeta(N_{d,n})} A^{N_{d,n}} \left(1 + O \left(\frac1{A} \right) \right).
\end{equation*}
Therefore, we see that
\begin{equation}
\label{Estimate S}
S_{d,n}(A,B) = \gamma_{d,n} \frac{\# \mathbb{V}_{d,n}(A)}{A} T_{d,n}(B)
\left( 1 + O \left( \frac{B^{1/(n+1-d)}}{A} \right) \right),
\end{equation}
where
\begin{equation*}
\gamma_{d,n} = \frac1{2} \cdot \frac{V_{N_{d,n} - 1}}{V_{N_{d,n}}} \cdot \frac{\zeta(N_{d,n})}{\zeta(N_{d,n}-1)}.
\end{equation*}

Our next task is to establish an asymptotic formula for the quantity $T_{d,n}(B)$. First, a M\"{o}bius inversion yields
\begin{equation*}
T_{d,n}(B) = \sum_{\ell \leq B^{1/(n+1-d)}} \frac{\mu(\ell)}{\ell^d}
\sum_{\substack{\mathbf{z} \in \mathbb{Z}^{n+1} \smallsetminus \{ \boldsymbol{0}\} \\ ||\mathbf{z}|| \leq B^{1/(n+1-d)}/\ell}} \frac1{||\nu_{d,n}(\mathbf{z})||}.
\end{equation*}
We now check that for $Z \geq 1$, we have
\begin{equation}
\label{Goal T(B)}
\sum_{\substack{\mathbf{z} \in \mathbb{Z}^{n+1} \smallsetminus \{ \boldsymbol{0}\} \\ ||\mathbf{z}|| \leq Z}}
\frac1{||\nu_{d,n}(\mathbf{z})||} = \left( \int_{\mathcal{B}_{n+1}(1)} \frac{\D \mathbf{t}}{||\nu_{d,n}(\mathbf{t})||} \right) Z^{n+1-d}
\left( 1 + O \left( \frac{\log Z}{Z} \right) \right),
\end{equation}
where the factor $\log Z$ in the error term is only necessary if $n = d$. Indeed, an application of partial summation gives
\begin{equation*}
\sum_{\substack{\mathbf{z} \in \mathbb{Z}^{n+1} \smallsetminus \{ \boldsymbol{0}\} \\ ||\mathbf{z}|| \leq Z}}
\frac1{||\nu_{d,n}(\mathbf{z})||} = \int_{0}^{\infty} \# \left\{ \mathbf{z} \in \mathbb{Z}^{n+1} \smallsetminus \{ \boldsymbol{0} \} :
\begin{array}{l l}
||\mathbf{z}|| \leq Z \\
||\nu_{d,n}(\mathbf{z})|| \leq u
\end{array}
\right\} \frac{\D u}{u^2}.
\end{equation*}
Furthermore, for $u \geq 1$ it follows from the work of Davenport \cite{MR43821} that 
\begin{align*}
\# \left\{ \mathbf{z} \in \mathbb{Z}^{n+1} \smallsetminus \{ \boldsymbol{0} \} :
\begin{array}{l l}
||\mathbf{z}|| \leq Z \\
||\nu_{d,n}(\mathbf{z})|| \leq u
\end{array}
\right\} = & \ \vol \left( \left\{ \mathbf{v} \in \mathbb{R}^{n+1} :
\begin{array}{l l}
||\mathbf{v}|| \leq Z \\
||\nu_{d,n}(\mathbf{v})|| \leq u
\end{array}
\right\} \right) \\
& + O \left( \min \left\{Z^n, u^{n/d} \right\} \right).
\end{align*}
In addition, if $u < 1$ we clearly have
\begin{align*}
\# \left\{ \mathbf{z} \in \mathbb{Z}^{n+1} \smallsetminus \{ \boldsymbol{0} \} :
\begin{array}{l l}
||\mathbf{z}|| \leq Z \\
||\nu_{d,n}(\mathbf{z})|| \leq u
\end{array}
\right\} = & \ \vol \left( \left\{ \mathbf{v} \in \mathbb{R}^{n+1} :
\begin{array}{l l}
||\mathbf{v}|| \leq Z \\
||\nu_{d,n}(\mathbf{v})|| \leq u
\end{array}
\right\} \right) \\
& + O \left( u^{(n+1)/d} \right).
\end{align*}
Therefore, we get
\begin{equation*}
\sum_{\substack{\mathbf{z} \in \mathbb{Z}^{n+1} \smallsetminus \{ \boldsymbol{0}\} \\ ||\mathbf{z}|| \leq Z}}
\frac1{||\nu_{d,n}(\mathbf{z})||} = \int_{\mathcal{B}_{n+1}(Z)} \left( \int_{||\nu_{d,n}(\mathbf{v})||}^{\infty} \frac{\D u}{u^2} \right)
\D \mathbf{v} + O \left( \mathcal{E}_{d,n}(Z) \right),
\end{equation*}
where
\begin{equation*}
\mathcal{E}_{d,n}(Z) = \int_{0}^{1} \frac{\D u}{u^{2-(n+1)/d}} + \int_{1}^{\infty} \min \left\{ Z^n, u^{n/d} \right\} \frac{\D u}{u^2}.
\end{equation*}
Moreover, an elementary calculation provides
\begin{equation*}
\mathcal{E}_{d,n}(Z) \ll
\begin{cases}
\log Z, & \textrm{if } n=d, \\
Z^{n-d}, & \textrm{if } n > d.
\end{cases}
\end{equation*}
We thus derive in particular
\begin{equation*}
\sum_{\substack{\mathbf{z} \in \mathbb{Z}^{n+1} \smallsetminus \{ \boldsymbol{0}\} \\ ||\mathbf{z}|| \leq Z}}
\frac1{||\nu_{d,n}(\mathbf{z})||} = \int_{\mathcal{B}_{n+1}(Z)} \frac{\D \mathbf{v}}{||\nu_{d,n}(\mathbf{v})||}
+ O \left( Z^{n-d} \log Z \right).
\end{equation*}
The equality \eqref{Goal T(B)} follows since the change of variables $\mathbf{v} = Z \mathbf{t}$ gives
\begin{equation*}
\int_{\mathcal{B}_{n+1}(Z)} \frac{\D \mathbf{v}}{||\nu_{d,n}(\mathbf{v})||} = \left( \int_{\mathcal{B}_{n+1}(1)} \frac{\D \mathbf{t}}{||\nu_{d,n}(\mathbf{t})||} \right) Z^{n+1-d}.
\end{equation*}
As a result, we deduce
\begin{equation*}
T_{d,n}(B) = \left(\int_{\mathcal{B}_{n+1}(1)} \frac{\D \mathbf{t}}{||\nu_{d,n}(\mathbf{t})||} \right) B 
\sum_{\ell \leq B^{1/(n+1-d)}} \frac{\mu(\ell)}{\ell^{n+1}} \left( 1 + O \left( \frac{\ell \log B}{B^{1/(n+1-d)}} \right) \right),
\end{equation*}
which eventually gives
\begin{equation*}
T_{d,n}(B) = \frac1{\zeta(n+1)} \left(\int_{\mathcal{B}_{n+1}(1)} \frac{\D \mathbf{t}}{||\nu_{d,n}(\mathbf{t})||} \right) B
\left( 1 + O \left( \frac{\log B}{B^{1/(n+1-d)}} \right) \right).
\end{equation*}

Recalling the estimate \eqref{Estimate S}, we see that we have proved that
\begin{equation*}
S_{d,n}(A,B) = C_{d,n} \# \mathbb{V}_{d,n}(A) \frac{B}{A}
\left( 1 + O \left( \frac{B^{1/(n+1-d)}}{A} + \frac{\log B}{B^{1/(n+1-d)}} \right) \right),
\end{equation*}
where
\begin{equation*}
C_{d,n} = \frac1{2 \zeta(n+1)} \cdot \frac{V_{N_{d,n} - 1}}{V_{N_{d,n}}} \cdot \frac{\zeta(N_{d,n})}{\zeta(N_{d,n}-1)} \cdot
\int_{\mathcal{B}_{n+1}(1)} \frac{\D \mathbf{t}}{||\nu_{d,n}(\mathbf{t})||},
\end{equation*}
which completes the proof of Theorem~\ref{Theorem 2}.
\end{proof}

\section{The smallest height of a rational point}

\label{Section smallest height}

In this section we show that Theorems~\ref{Theorem 1} and~\ref{Theorem 1'} follow from Theorem~\ref{Theorem 2}. We start with the proof of Theorem~\ref{Theorem 1}.

\begin{proof}[Proof of Theorem~\ref{Theorem 1}]
Our aim is to prove that
\begin{equation}
\label{Goal Theorem 1}
\lim_{A \to \infty} \frac{\# \left\{ V \in \mathbb{V}_{d,n}(A) : \mathfrak{M}(V) \leq \psi(||\mathbf{a}_V||) \right\}}
{\# \mathbb{V}_{d,n}(A)} = 0.
\end{equation}
Let $\eta \in (0,1)$. By assumption, if $A$ is sufficiently large then for any $u \geq A^{1/2}$, the inequality
$\psi(u) \leq \eta u$ holds. Since we have
\begin{align*}
\# \left\{ V \in \mathbb{V}_{d,n}(A) : \mathfrak{M}(V) \leq \psi(||\mathbf{a}_V||) \right\} = & \ 
\# \left\{ V \in \mathbb{V}_{d,n}(A) :
\begin{array}{l l}
||\mathbf{a}_V|| > A^{1/2} \\
\mathfrak{M}(V) \leq \psi(||\mathbf{a}_V||) 
\end{array} \right\} \\
& + O\left(A^{N_{d,n}/2}\right),
\end{align*}
we deduce
\begin{equation*}
\# \left\{ V \in \mathbb{V}_{d,n}(A) : \mathfrak{M}(V) \leq \psi(||\mathbf{a}_V||) \right\} \ll
\# \left\{ V \in \mathbb{V}_{d,n}(A) : \mathfrak{M}(V) \leq \eta A \right\} + A^{N_{d,n}/2}.
\end{equation*}
Moreover, we clearly have
\begin{equation*}
\# \left\{ V \in \mathbb{V}_{d,n}(A) : \mathfrak{M}(V) \leq \eta A \right\} \leq
\sum_{ V \in \mathbb{V}_{d,n}(A)} N_V(\eta A).
\end{equation*}
Since $A \geq (\eta A)^{1/(n+1-d)}$ we can apply Theorem~\ref{Theorem 2}. We deduce in particular that
\begin{equation*}
\sum_{ V \in \mathbb{V}_{d,n}(A)} N_V(\eta A) \ll \eta \cdot \# \mathbb{V}_{d,n}(A).
\end{equation*}
As a result, we eventually obtain
\begin{equation*}
\limsup_{A \to \infty} \frac{\# \left\{ V \in \mathbb{V}_{d,n}(A) : \mathfrak{M}(V) \leq \psi(||\mathbf{a}_V||) \right\}}{\# \mathbb{V}_{d,n}(A)} \ll \eta.
\end{equation*}
This upper bound holds for any $\eta \in (0,1)$ and the equality \eqref{Goal Theorem 1} thus follows, which completes the proof of
Theorem~\ref{Theorem 1}.
\end{proof}

We now furnish the proof of Theorem~\ref{Theorem 1'}.

\begin{proof}[Proof of Theorem~\ref{Theorem 1'}]
It follows from the work of Davenport \cite{MR86105} that for any $V \in \mathbb{V}_{2,2}(A)$, we have
$\mathfrak{M}(V) \leq 3 ||\mathbf{a}_V||$ if and only if $V(\mathbb{Q}) \neq \emptyset$. Therefore, we observe that our aim is to prove that
\begin{equation}
\label{Goal Theorem 1'}
\lim_{A \to \infty} \frac{\# \left\{ V \in \mathbb{V}_{2,2}(A) : \mathfrak{M}(V) \leq \psi(||\mathbf{a}_V||) \right\}}
{\# \left\{ V \in \mathbb{V}_{2,2}(A) : V(\mathbb{Q}) \neq \emptyset \right\}} = 0.
\end{equation}
Let $\eta \in (0,1)$. Proceeding as in the proof of Theorem~\ref{Theorem 1}, we obtain
\begin{equation}
\label{Upper bound 1}
\# \left\{ V \in \mathbb{V}_{2,2}(A) : \mathfrak{M}(V) \leq \psi(||\mathbf{a}_V||) \right\} \ll
\sum_{ V \in \mathbb{V}_{2,2}(A)} N_V \left(\eta \frac{A}{(\log A)^{1/2}} \right) + A^3.
\end{equation}
Moreover, we have $A \geq \eta A/(\log A)^{1/2}$ so we can apply Theorem~\ref{Theorem 2}. We deduce in particular that
\begin{equation}
\label{Upper bound 2}
\sum_{ V \in \mathbb{V}_{2,2}(A)} N_V \left(\eta \frac{A}{(\log A)^{1/2}} \right) \ll
\eta \frac{\# \mathbb{V}_{2,2}(A)}{(\log A)^{1/2}}.
\end{equation}

In addition, the result of Hooley \cite[Theorem]{MR2300456} can be rewritten as
\begin{equation*}
\sum_{\substack{V \in \mathbb{V}_{2,2}(A) \\ V(\mathbb{Q}) \neq \emptyset}} \left\lfloor \frac{A}{||\mathbf{a}_V||} \right\rfloor \gg
\frac{\# \mathbb{V}_{2,2}(A)}{(\log A)^{1/2}}.
\end{equation*}
For $\varepsilon \in (1/A^{1/2},1)$, we split the summation over $V$ depending on whether $||\mathbf{a}_V|| \leq \varepsilon A$ or
$||\mathbf{a}_V|| > \varepsilon A$. We get
\begin{equation*}
\frac{\# \mathbb{V}_{2,2}(A)}{(\log A)^{1/2}} \ll A
\sum_{\substack{V \in \mathbb{V}_{2,2}(\varepsilon A) \\ V(\mathbb{Q}) \neq \emptyset}} \frac1{||\mathbf{a}_V||} + \frac1{\varepsilon} \cdot \# \left\{ V \in \mathbb{V}_{2,2}(A) \smallsetminus \mathbb{V}_{2,2}(\varepsilon A) : V(\mathbb{Q}) \neq \emptyset \right\}.
\end{equation*}
Furthermore, an application of partial summation yields
\begin{align*}
\sum_{\substack{V \in \mathbb{V}_{2,2}(\varepsilon A) \\ V(\mathbb{Q}) \neq \emptyset}} \frac1{||\mathbf{a}_V||} = & \ 
\frac1{\varepsilon A} \cdot \# \left\{ V \in \mathbb{V}_{2,2}(\varepsilon A) : V(\mathbb{Q}) \neq \emptyset \right\} \\
& + \int_1^{\varepsilon A} \frac{\# \left\{ V \in \mathbb{V}_{2,2}(t) : V(\mathbb{Q}) \neq \emptyset \right\}}{t^2} \D t.
\end{align*}
We thus obtain
\begin{equation*}
\frac{\# \mathbb{V}_{2,2}(A)}{(\log A)^{1/2}} \ll
\frac1{\varepsilon} \cdot \# \left\{ V \in \mathbb{V}_{2,2}(A) : V(\mathbb{Q}) \neq \emptyset \right\}
+ A \int_1^{\varepsilon A} \frac{\# \left\{ V \in \mathbb{V}_{2,2}(t) : V(\mathbb{Q}) \neq \emptyset \right\}}{t^2} \D t.
\end{equation*}
Moreover, it follows from the upper bound of Serre \cite[Exemple $4$]{MR1075658} that
\begin{equation*}
A \int_1^{\varepsilon A} \frac{\# \left\{ V \in \mathbb{V}_{2,2}(t) : V(\mathbb{Q}) \neq \emptyset \right\}}{t^2} \D t \ll 
\varepsilon^5 \frac{\# \mathbb{V}_{2,2}(A)}{(\log A)^{1/2}}.
\end{equation*}
Hence, we get
\begin{equation*}
\frac{\# \mathbb{V}_{2,2}(A)}{(\log A)^{1/2}} \ll
\frac1{\varepsilon} \cdot \# \left\{ V \in \mathbb{V}_{2,2}(A) : V(\mathbb{Q}) \neq \emptyset \right\}
+ \varepsilon^5 \frac{\# \mathbb{V}_{2,2}(A)}{(\log A)^{1/2}}.
\end{equation*}
As a result, if $A$ is sufficiently large then by choosing $\varepsilon$ small enough independently of $A$, we derive
\begin{equation*}
\# \left\{ V \in \mathbb{V}_{2,2}(A) : V(\mathbb{Q}) \neq \emptyset \right\} \gg \frac{\# \mathbb{V}_{2,2}(A)}{(\log A)^{1/2}}.
\end{equation*}

Recalling the upper bounds \eqref{Upper bound 1} and \eqref{Upper bound 2}, we eventually deduce
\begin{equation*}
\limsup_{A \to \infty} \frac{\# \left\{ V \in \mathbb{V}_{2,2}(A) : \mathfrak{M}(V) \leq \psi(||\mathbf{a}_V||) \right\}}
{\# \left\{ V \in \mathbb{V}_{2,2}(A) : V(\mathbb{Q}) \neq \emptyset \right\}} \ll \eta.
\end{equation*}
This upper bound is valid for any $\eta \in (0,1)$ so the equality \eqref{Goal Theorem 1'} follows, which completes the proof of
Theorem~\ref{Theorem 1'}.
\end{proof}

\bibliographystyle{amsalpha}
\bibliography{biblio}

\end{document}